\documentclass[11pt, a4paper]{article}
\usepackage{fullpage}
\usepackage{graphicx}
\usepackage{tikz}
\usepackage{amssymb,amsmath,amsthm}
\usepackage{verbatim}
\usepackage{color}
\usepackage[shortlabels]{enumitem}
\usepackage{dsfont}
\usepackage[margin=2.75cm]{geometry}
\usepackage{thmtools, thm-restate}
\usepackage[colorlinks=true]{hyperref}
\hypersetup{citecolor=blue, linkcolor=blue}
\usepackage[capitalise]{cleveref}

\declaretheorem{theorem}
\declaretheorem[sibling=theorem]{corollary, lemma, proposition, question, definition, conjecture,remark,example}

\def\F{\mathbb{F}}

\title{Sphere intersections and incidences over finite fields}

\author{Doowon Koh \thanks{Department of Mathematics, Chungbuk National University, Cheongju, Chungbuk 28644 Korea. Email: {\tt koh131@chungbuk.ac.kr
}} \and Ben Lund\thanks{Institute of Mathematics and Interdisciplinary Studies, Xidian University, Xi’an, 710071, China. Email: {\tt lund.ben@gmail.com}} \and Chuandong Xu \thanks{School of Mathematics and Statistics, Xidian University, Xi’an, 710071, China. Email: {\tt xuchuandong@xidian.edu.cn}} \and Semin Yoo\thanks{Discrete Mathematics Group, Institute for Basic Science (IBS). Email: {\tt syoo19@ibs.re.kr}}}

\begin{document}

\maketitle

\begin{abstract}
We bound the number of incidences between points and spheres in finite vector spaces by bounding the sum of the number of points in the pairwise intersections of the spheres. 
We obtain new incidence bounds that are interesting when the number of spheres is not too large.
Our approach also leads to an elementary proof of the Iosevich-Rudnev bound on the Erd\H{o}s-Falconer distance problem in odd dimensions.

\end{abstract}

\section{Introduction}
Throughout this paper, let $q$ be an odd prime power, $\F_q$ the finite field with $q$ elements, and $\F_q^{\times}=\F_q \setminus \{0\}$.
For $x \in \F_q^d$, let $\|x\| = x_1^2 + x_2^2 + \cdots + x_d^2$.
The set $s = \{x \in \F_q^d: \|x-y\| = r\}$ is the $\|\cdotp\mkern-3mu\|$-sphere with center $y \in \F_q^d$ and radius $r \in \F_q$, and we omit $\|\cdotp\mkern-3mu\|$ when it is obvious from context.
A point $x$ is incident to a sphere $s$ if $x \in s$, and for $P \subseteq \F_q^d$ and $S$ a set of spheres in $\F_q^d$, we denote by
\[I(P,S) = \#\{(x,s) \in P \times S: x \in s\}\]
the number of incidences between $P$ and $S$.

Bounds on the number of point sphere incidences over finite fields have been the subject of several papers, and such bounds have a number of important applications to problems in finite geometry and additive combinatorics over finite fields; for example, see \cite{cilleruelo2014elementary,phuong2017incidences,koh2022finite,koh2022point,chakraborti2024almost}.

The first and most important result on this subject is implicit in the work of Iosevich and Rudnev on the Erd\H{o}s-Falconer distance problem \cite{iosevich2007erdos}:

\begin{theorem}[\cite{iosevich2007erdos}]\label{thm:iosevichRudnev}
    For $d \geq 2$, let $S$ be a set of $\|\cdotp\mkern-3mu\|$-spheres, each with the same radius $r \neq 0$, and let $P \subseteq \mathbb{F}_q^d$.
    Then,
    \[|I(P,S) - q^{-1}|P|\,|S| |\leq 2 \sqrt{q^{d-1}|P|\,|S|}. \]
\end{theorem}

The assumption in \cref{thm:iosevichRudnev} that all of the radii are the same can be relaxed using the Cauchy-Schwarz inequality, as follows:

\begin{corollary}\label{th:iosevichRudnevMultipleRadii}
    Let $R \subseteq \F_q^\times$, let $S$ be a set of $\|\cdotp\mkern-3mu\|$-spheres with radii in $R$, and let $P \subseteq \F_q^d$.
    Then,
    \[|I(P,S) - q^{-1}|P|\,|S|| \leq 2 \sqrt{q^{d-1}|R|\,|P|\,|S|}. \]
\end{corollary}
\begin{proof}
    For $r \in R$, let $S_r = \{s \in S: r(s) = r\}$.
    \begin{align*}
        \left(I(P,S) - q^{-1}|P|\,|S| \right)^2 &= \left(\sum_{r \in R} \Big(I(P,S_r) - q^{-1}|P|\,|S_r| \Big)
        \right)^2 \\
        &\leq |R| \sum_{r \in R}\left( I(P,S_r) - q^{-1}|P|\,|S_r| \right)^2 \\
        &\leq |R| \sum_{r \in R} 4 q^{d-1}|P|\,|S_r| \\
        &= 4q^{d-1}|R|\,|P|\,|S|.
    \end{align*}
\end{proof}

The extension of \cref{th:iosevichRudnevMultipleRadii} to the case $R=\F_q$ was proved independently in \cite{cilleruelo2014elementary,phuong2017incidences}, with an improved constant term:
\begin{theorem}\label{th:simpleBound}
    Let $S$ be a set of $\|\cdotp\mkern-3mu\|$-spheres (with arbitrary radii), and let $P \subseteq \F_q^d$.
    Then,
    \[|I(P,S) - q^{-1}|P|\,|S|| \leq \sqrt{q^{d}|P|\,|S|}. \]
\end{theorem}
It seems that \cref{th:simpleBound} is much easier to prove than \cref{thm:iosevichRudnev}.
The only known proof of \cref{thm:iosevichRudnev} in full generality depends on Weil's bounds on Kloosterman sums over finite fields \cite[Chapter 5.5]{LN97}.
By contrast, the proof of \cref{th:simpleBound} in \cite{cilleruelo2014elementary} is self-contained and elementary, depending on straightforward ideas from additive combinatorics.
Two of the results in this paper (\cref{th:newIosevichRudnev,th:simpleWeightedBound}) are a new, simple, proof of \cref{th:simpleBound} that gives some intuition for this seeming difference, and an elementary proof of the odd dimension case of \cref{thm:iosevichRudnev} that avoids the use of Kloosterman sums.

In general, \cref{th:simpleBound} is tight, as shown by the following constructions.

\begin{example}\label{const:simple}
For the simplest example, let $P=\{x\}$, and let $S$ be the set of all $q^d$ spheres that contain $x$.
Then, $I(P,S) = q^d = \sqrt{q^d|P|\|S|}$.
\end{example}

A more interesting example showing the optimality of \cref{th:simpleBound} occurs only in odd dimensions.
We describe this example in detail in \cref{sec:quadraticForms}, after introducing our notation for general quadratic forms.
The example we describe is essentially identical to \cite[Proposition 1.1]{kang2025erd}, appeared implicitly in \cite{iosevich2023quotient}, and generalizes an example first noticed in \cite{hart2011averages}.

\begin{example}\label{const:zeroRadEvenDim}
Another construction occurs when $d \equiv 0 \mod 4$ or when $d \equiv 2 \mod 4$ and $q \equiv 1 \mod 4$.
In this case, there are $d/2$-dimensional isotropic subspaces, hence there are sets of $|P|=q^{d/2}$ points such that $\|x-y\|=0$ for each pair $x,y \in P$.
If we take $S$ to be the set of all spheres whose centers are the points of $P$ and whose radii are zero, then $|S| = |P| = q^{d/2}$, and $I(P,S) = q^d = \sqrt{q^d|P|\,|S|}$.
\end{example}

Although \cref{th:simpleBound} is tight in general, we can hope to improve it in certain cases.
The most interesting case is for $d$ even, when $|P|$ is roughly $q^{(d+1)/2}$ and $|S|$ is roughly $q^{(d+3)/2}$, because this is the critical range to improve the bounds for the Erd\H{o}s-Falconer distance problem; see for example \cite{iosevich2007erdos,hart2011averages,chapman2012pinned,hanson2016distinct,lund2020bisectors,murphy2022pinned,bennett2013group,koh2015distance,covert2019generalized} for discussion of this problem.
In this paper, we prove new incidence bounds in the case that $|S|$ is smaller than this critical range.

Instead of considering only the specific quadratic form $\|x\| = x_1^2 + x_2^2 + \cdots + x_d^2$, we find it more convenient to consider the slightly more general setting of arbitrary non-degenerate quadratic forms.
Before stating our results, we introduce the relevant notation and review some basic facts about quadratic forms over finite fields of odd order.

\subsection{Quadratic forms}\label{sec:quadraticForms}
We refer to \cite[section 6.2]{LN97} for the basic theory of quadratic forms over finite fields of odd characteristic.
We use $\eta$ to denote the quadratic character over $\F_q$, so $\eta(x)=0$ if $x=0$, $\eta(x)=1$ if $x$ is a square, and $\eta(x)=-1$ if $x$ is not a square.

We can write a quadratic form $f$ as
\[f(x) = \sum_{i,j=1}^n a_{ij}x_ix_j, \text{ with }a_{ij}=a_{ji}. \]
Using this representation, we associate $f$ with a symmetric $n \times n$ matrix $A$ whose $(i,j)$ entry is $a_{ij}$.
We say that $f$ is non-degenerate if $A$ is non-singular.
Quadratic forms $f$ and $g$ associated with matrices $A$ and $B$, respectively, are {\em equivalent} if there is a non-singular matrix $C$ such that $B=C^TAC$.
It is a standard result that there are two inequivalent nondegenerate quadratic forms in each dimension, depending on whether the determinant of the matrix associated to the form is a quadratic residue or non-residue. 
In this paper, we use the following representative forms:
\begin{align*}
    Q_1^d(x) &= x_1x_2 + x_3x_4 + \cdots + x_{d-1}x_d, \\
    Q_2^d(x) &= x_1x_2 + x_3x_4 + \cdots + x_{d-3}x_{d-2}+ x_{d-1}^2 - \epsilon x_d^2, \\
    Q_3^d(x) &= x_1x_2 + x_3x_4 + \cdots + x_{d-2}x_{d-1} - x_d^2, \\
    Q_4^d(x) &= x_1x_2 + x_3x_4 + \cdots + x_{d-2}x_{d-1} - \epsilon x_d^2,
\end{align*}
where $\epsilon$ is a non-square, and $d$ must be even for $Q_1,Q_2$, and $d$ must be odd for $Q_3,Q_4$.
We omit $d$ and $i$ from the notation when it is obvious from context.
We refer to the translates of the level sets of a quadratic form $Q$ as \textit{$Q$-spheres}.
In more detail,
$s=\{x : Q_i(x-y)=r\}$
is the $Q_i$-sphere with center $c(s)=y$ and radius $r(s)=r$.

Note that 
\begin{align*}
    \eta(\det(Q_1^d)) &= \eta((-1)^{d/2}), &
    \eta(\det(Q_2^d)) &= -\eta((-1)^{d/2}), \\
    \eta(\det(Q_3^d)) &= \eta((-1)^{(d+1)/2}), &
    \eta(\det(Q_4^d)) &= -\eta((-1)^{(d+1)/2}),
\end{align*}
where $\det(Q)$ is the determinant of the matrix associated with $Q$.
Since $-1$ is a square in $\F_q$ if and only if $q \equiv 1 \mod 4$, and $\det(\|\cdotp\mkern-3mu\|)=1$, we have:
\begin{itemize}
    \item For even $d$, the form $\|\cdotp\mkern-3mu\|$ is equivalent to $Q_2^d$ if and only if $q \equiv 3 \mod 4$ and $d \equiv 2 \mod 4$; otherwise, it is equivalent to $Q_1^d$.
    \item For odd $d$, the form $\|\cdotp\mkern-3mu\|$ is equivalent to $Q_4^d$ if and only if $q \equiv 3 \mod 4$ and $d \equiv 1 \mod 4$; otherwise, it is equivalent to $Q_3^d$.
\end{itemize}

Here comes the example of \cite[Proposition 1.1]{kang2025erd}, restated in our notation.
The example shows that \cref{th:simpleBound} is within a constant factor of being tight in odd dimensions, even in the critical range for the Erd\H{o}s-Falconer distance problem.
\begin{example}\label{const:oddDim}
Let $d$ be odd, let $Q = Q_i^d$ for $i \in \{3,4\}$, and let
\[P = \{(0,x_1,0,x_2,\ldots,0,x_{(d-1)/2},y)  : x_1,\ldots,x_{(d-1)/2},\ y \in \F_q\}.\]
It is straightforward to check that, for $a,b \in P$, we have $Q(a-b)=Q((0,\ldots,0,z))$ for some $z \in \F_q$.
Hence, $\eta(Q(a-b)) = (-1)^{i+1}\eta(-1)$.
Let $S$ be the set of $Q$-spheres with centers in $P$ and radii in $\{r:\eta(r) = (-1)^{i+1}\eta(-1)\}$.
Then, $|P|=q^{(d+1)/2}$, and $|S| = \frac{q-1}{2}|P|$, and $I(P,S) = |P|(|P|-1) = q^{d+1}-q^{(d+1)/2} = (2^{1/2}+o(1))\sqrt{q^d|P|\,|S|}$.
\end{example}

\subsection{Results}

The main contribution of this paper is to bound point-sphere incidences by bounding sphere-sphere intersections.
This builds on a standard idea in the study of pseudo-random graphs \cite{thomason1987pseudo, thomason1989dense}, which is that good expansion properties of a graph can be derived from a good bound on the intersection of the neighborhoods of vertices.

We use this approach to obtain a new proof of the odd dimension case of \cref{thm:iosevichRudnev}, and a new proof of \cref{th:simpleBound}.
The new proof of \cref{th:simpleBound} is very simple, and suggests a reason for the observation that \cref{th:simpleBound} seems to be much easier to prove than \cref{thm:iosevichRudnev}; see the discussion after the proof of \cref{th:simpleWeightedBound} in \cref{sec:combinatorial}.

Our new proof of \cref{thm:iosevichRudnev} for odd $d$ is elementary and self-contained - see \cref{th:newIosevichRudnev}.
The problem of obtaining an elementary proof of \cref{thm:iosevichRudnev} has been previously discussed in the literature, for example \cite{koh2022point}.

We also obtain new point-sphere incidence bounds that are interesting in the case that there are not too many spheres.
In even dimensions, we prove:
\begin{restatable}{theorem}{smallSBound}\label{th:smallSBound}
    Let $d$ be even, let $i \in \{1,2\}$, and $Q = Q_i^d$.
    Let $P$ be a set of points in $\F_q^d$.
    If $S$ a set of non-zero-radius $Q$-spheres, then
    \begin{equation}\label{eq:i2IncidenceBound}
    \left| I(P,S) - q^{-1}\,|P|\,|S| \right| \leq (1+o(1))\sqrt{|P|\Big(q^{d-1} |S| + q^{d/2-1}|S|^2\Big)}. \end{equation}
    If $i=2$ and $S$ is a set of zero-radius $Q$-spheres, then
    \begin{equation}\label{eq:i2ZeroIncidenceBound}
        \left| I(P,S) - q^{-1}\,|P|\,|S| \right| \leq (1+o(1))\sqrt{|P|\Big(q^{d-1} |S| + q^{d/2-1}|S|^2\Big)}.
    \end{equation}
    If $i=1$ and $S$ is a set of zero-radius $Q$-spheres, then
    \begin{equation} \label{eq:i1IncidenceBound}
    \left| I(P,S) - q^{-1}\,|P|\,|S| \right| \leq (1+o(1))\sqrt{|P|\Big(q^{d-1} |S| + q^{d/2}|S|^2\Big)}. \end{equation}
\end{restatable}
\cref{eq:i2IncidenceBound} improves the bound in \cref{thm:iosevichRudnev} for $|S| < q^{d/2}$, and improves the bound in \cref{th:simpleBound} for $|S| < q^{d/2 + 1}$.
Although the bound of \cref{eq:i1IncidenceBound} is much weaker than \cref{eq:i2IncidenceBound} or \cref{eq:i2ZeroIncidenceBound}, it is tight for $|S|=|P|=q^{d/2}$, as shown in \cref{const:zeroRadEvenDim}.
If we could replace the term $q^{d/2 - 1}|S|^2$ on the right hand side of \cref{eq:i2IncidenceBound} by $q^{d/2 - 3/4}|S|^2$, it would improve \cref{th:simpleBound} for all $|S| < q^{(d+3)/2}$.
Improvement beyond this point would lead to improved bounds for the Erd\H{o}s-Falconer distance problem in even dimensions.

Certain cases of \cref{th:smallSBound} were previously proved by Koh, Lee, and Pham \cite{koh2022finite}, who relied on a restriction estimate for cones in four dimensions.
In particular, they proved:
\begin{theorem}[Theorem 1.6, \cite{koh2022finite}]\label{th:oldEvenBound}
    Let $P$ be a set of points in $\F_q^d$, and let $S$ be a set of $\|\cdotp\mkern-3mu\|$-spheres in $\F_q^d$.
    \begin{enumerate}
        \item If $d \equiv 2 \mod 4$ and $q \equiv 3 \mod 4$ and $|S| \leq q^{d/2}$, then
        \[\left| I(P,S) - q^{-1}\,|P|\,|S| \right| \ll q^{(d-1)/2}|P|^{1/2}|S|^{1/2}. \]
        \item If $d \equiv 2 \mod 4$ and $q \equiv 1 \mod 4$, or if $d \equiv 0 \mod 4$, then the same conclusion holds under the assumption that $|S| \leq q^{(d-2)/2}$.
    \end{enumerate}
\end{theorem}
The notation $a \ll b$ means that, for each fixed $d$, there exists a constant $c_d$ such that $a \leq c_d b$ for all $q$.

\Cref{th:oldEvenBound} is a special case of \cref{th:smallSBound}, in the case that $S$ is a small set of spheres of arbitrary radius (possibly including zero), and $Q=\|\cdot\|$.
In particular, the first enumerated item in \cref{th:oldEvenBound} corresponds to the case $i=2$ in \cref{th:smallSBound}, and the second corresponds to \cref{eq:i1IncidenceBound}.

Comparing the two theorems, we can see that \cref{th:smallSBound} improves \cref{th:oldEvenBound} in a number of ways:
\begin{itemize}
    \item we improve the bound of \cref{th:simpleBound} for larger $|S|$,
    \item our bound replaces an implicit constant depending on dimension by $(1+o(1))$,
    \item since we explicitly differentiate between spheres of zero and non-zero radius, we are able to prove a much stronger bound for $Q_1^d$-spheres of non-zero radius,
    \item we include the case $i=2$, $d \equiv 0 \mod 4$, which is missing from \cref{th:oldEvenBound} since this quadratic form is never equivalent to $\|x\|$.
\end{itemize}

In odd dimensions, we prove the following new incidence bound:
\begin{restatable}{theorem}{ourOddBound}\label{th:ourOddBound}
    Let $d$ be odd, let $i \in \{3,4\}$.
    Let $R \subseteq \F^\times_q$ such that $(-1)^{i+1}\eta(-r) = -1$ for each $r \in R$.
    Let $P \subseteq \mathbb{F}_q^d$, and let $S$ be a set of $Q_i^d$-spheres with radii in $R$.
    Then,
    \[\left|I(P,S) - q^{-1}|P|\,|S|\right| \leq q^{(d-1)/2}|P|^{1/2}|S|^{1/2} + \sqrt{3}q^{(d-3)/4}|P|^{1/2}|S|. \]
\end{restatable}

Note that $q^{(d-1)/2}|P|^{1/2}|S|^{1/2} > q^{(d-3)/4}|P|^{1/2}|S|$ when $|S| < q^{(d+1)/2}$, so we obtain the bound of \cref{thm:iosevichRudnev} for spheres with many different radii when this holds.
We compare \cref{th:ourOddBound} to the following bound, proved by Koh and Pham \cite{koh2022point}:
\begin{theorem}[Theorems 1.4, 1.7 \cite{koh2022point}]\label{th:oldOddBound}
    Let $d$ be odd, let $P \subseteq \F_q^d$, and let $S$ be a set of $\|\cdotp\mkern-3mu\|$-spheres in $\F_q^d$, and suppose that $|P|,|S| \leq N$.
    \begin{enumerate}
        \item If $d \equiv 3 \mod 4$, $q \equiv 3 \mod 4$, and the spheres in $S$ all have square radii, then
        \[I(P,S) \ll q^{-1}N^2 + q^{(d-1)/2}N. \]
        \item If $d \equiv 1 \mod 4$, $q \equiv 3 \mod 4$, and the spheres in $S$ all have non-square radii, then the same conclusion holds.
    \end{enumerate}
\end{theorem}
\cref{th:ourOddBound} is neither a generalization of nor a special case of \cref{th:oldOddBound}.
\cref{th:oldOddBound} is stronger than \cref{th:ourOddBound} in the case that $|P| = |S| = N$ and $N$ is much larger than $q^{(d+1)/2}$.
On the other hand, \cref{th:ourOddBound} is quantitatively stronger than \cref{th:oldOddBound} when $|S| \leq |P|$ and $|S| \leq q^{(d+1)/2}$.
In addition, \cref{th:ourOddBound} includes a number of cases that \cref{th:oldOddBound} misses, including some that are relevant for $\|x\|$.
In particular, if $q \equiv 1 \mod 4$ and the spheres in $S$ all have square radii, then \cref{th:ourOddBound} applies to the quadratic form $\|x\| = Q_3^d(x)$.

\subsection{Organization}
In \cref{sec:combinatorial}, we introduce our graph-theoretic tool, and give a new proof of \cref{th:simpleBound}.
\cref{sec: int of spheres} recalls known results on the number of points in the pairwise intersections of spheres.
\cref{sec: new incidences} presents a new proof of \cref{thm:iosevichRudnev} in the odd dimension case, and includes the proofs of \cref{th:smallSBound} and \cref{th:ourOddBound}.

\section{Pseudo-random graphs}\label{sec:combinatorial}

For any graph $G$ and vertices $u,v \in V(G)$, we write $u \sim v$ if $\{u,v\} \in E(G)$.
For subsets $L,R \subseteq V(G)$, we define
\[E(L,R) = \#\{(u,v) \in L \times R: u \sim v\}. \]
We rely on the following bound, which is a mild generalization of \cite[Theorem 2]{thomason1989dense}:
\begin{theorem}\label{thm:tomason}
    Let $G$ be a bipartite graph with vertex classes $L$ and $U$, and let $p$ be a positive real with $|L|^{-1} \leq p \leq |L|^{-1}|U|^{-1}|E(L,U)|$.
    Then, for any $R \subseteq U$,
    \[\left(E(L,R) - p|L|\,|R| \right)^2 \leq |R|\left(\sum_{u \in U}\sum_{v_1 \neq v_2 \in L}(\mathbf{1}_{v_1 \sim u}\mathbf{1}_{v_2 \sim u} - p^2) + E(L,U) \right). \]
\end{theorem}
\begin{proof}
Using the Cauchy-Schwarz inequality,
\begin{align*}
    |R|^{-1} &\left(\sum_{u \in R} \sum_{v \in L} (\mathbf{1}_{v \sim u} - p) \right)^2 \\
    &\leq \sum_{u \in R} \left (\sum_{v \in L} (\mathbf{1}_{v \sim u} - p) \right)^2\\
    &\leq \sum_{u \in U} \left( \sum_{v \in L}(\mathbf{1}_{v \sim u} - p) \right)^2 \\
    &= \sum_{u \in U} \sum_{v_1 \neq v_2 \in L}\mathbf{1}_{v_1 \sim u}\mathbf{1}_{v_2 \sim u} + E(L,U) - 2p|L|E(L,U)+p^2|L|^2|U| \\
    &\leq \sum_{u \in U} \sum_{v_1 \neq v_2 \in L}(\mathbf{1}_{v_1 \sim u}\mathbf{1}_{v_2 \sim u} - p^2) + E(L,U).
\end{align*}
\end{proof}

For sets $P$ of points and $S$ of spheres, denote
\[
I(P,S) = \sum_{x \in P}\sum_{s \in S}\mathbf{1}_{x \in s}. 
\]

Here comes the new proof of \cref{th:simpleBound}, which works for arbitrary nondegenerate quadratic forms.

\begin{theorem}\label{th:simpleWeightedBound}
    Let $Q$ be a nondegenerate quadratic form on $\mathbb{F}_q^d$.
    Let $P \subseteq \F_q^d$, and let $S$ be a set of $Q$-spheres.
    Then,
    \begin{equation}\label{eq:simpleBoundInequality}
        |I(P,S)- q^{-1} \, |P|\, |S|| \leq \sqrt{q^d|P||S|}.
    \end{equation}
\end{theorem}
\begin{proof}
    We will apply \cref{thm:tomason} with $L=P$, $R=S$, and $U$ the set of all $q^{d+1}$ spheres in $\F_q^d$.
    Each point $x$ is contained in exactly $q^d$ spheres: for each point $y \in \F_q^d$, there is a unique sphere centered $y$ that contains $x$.
    Hence, $I(P,U) = q^d|P|$ and we may choose $p=q^{-1}$.
    \cref{thm:tomason} yields
    \[\left(I(P,S) - q^{-1}|P|\,|S|\right)^2 \leq |S| \left(\sum_{s \in U} \sum_{x\neq y \in P} (\mathbf{1}_{x \in s}\mathbf{1}_{y \in s} - p^2) + q^{d}|P|\right). \]
    For each pair of distinct points $x \neq y \in \F_q^d$, the set $B(x,y)=\{z:Q(z-x)=Q(z-y)\}$ is a hyperplane.
    This is easy to check directly from the definition of $Q$: the quadratic terms of $Q(z-x)$ and $Q(z-y)$ are equal, and the linear terms are equal if and only if $x=y$.
    Hence, 
    \[\sum_{s \in U} \sum_{x\neq y \in P} \mathbf{1}_{x \in s}\mathbf{1}_{y \in s} = |P|(|P|-1)q^{d-1} = p^2|U|\,|P|(|P|-1),\]
    so $\sum_{s \in U} \sum_{x\neq y \in P} (\mathbf{1}_{x \in s}\mathbf{1}_{y \in s} - p^2) = 0$, which completes the proof.
\end{proof}

The basic geometric fact that enables this simple proof of \cref{th:simpleBound} is that each pair $(x,y)$ of distinct points in $\F_q^d$ is contained in exactly $q^{d-1}$ distinct spheres.
We obtain a $q^{d/2}$ factor on the right hand side of \cref{eq:simpleBoundInequality} since we extend the sum over all $q^{d+1}$ spheres.
We can hope to obtain the better $q^{(d-1)/2}$ factor in some cases (as in \cref{thm:iosevichRudnev}) by extending instead over the points.
However, this is complicated by the fact that the number of points in the intersection of two spheres depends on the radii of the spheres and the distance between their centers.
Hence, obtaining concrete bounds using this approach requires a finer analysis, which depends on the specific quadratic form and collection of spheres that we consider.

Here is the general bound that we rely on, which requires further analysis of the intersections of the spheres before it can be applied.
Note that we always apply the bound to collections of spheres having the same size.

\begin{theorem}\label{th:generalWeightedBound}
    Let $Q$ be a nondegenerate quadratic form on $\F_q^d$.
    Let $P$ be a set of points in $\F_q^d$, and $S$ a set of spheres such that, for each $s \in S$, the number of points in $s$ is $|s|=q^{d-1}+N$ for a fixed constant $N$.
    Let $p = q^{-1}+q^{-d}N$ be the probability that a uniformly random point in $\F_q^d$ is contained in a fixed sphere of $S$.
    Then,
    \[\left(I(P,S) - p\,|P|\,|S| \right)^2 \leq |P|\sum_{s_1 \neq s_2 \in S} \left(|s_1 \cap s_2| - q^dp^2 \right) + |P||S|(q^{d-1}+N). \]
\end{theorem}
\begin{proof}
    This follows directly from \cref{thm:tomason}, taking $L=S$, $R=P$, and $U = \F_q^d$.
\end{proof}

\section{Pairwise intersections of spheres}\label{sec: int of spheres}

In order to apply \cref{th:generalWeightedBound} in specific cases, we need to understand the sizes and pairwise intersections of spheres in $\F_q^d$.
This is completely standard and well studied.

The size of a sphere $s$ depends only on $Q$ and the quadratic character $\eta(-r)$ where $r$ is the radius of $s$, and is equal to $q^{d-1} + N(Q,\eta(-r))$, where $N$ is a function that takes a quadratic form and integer in $\{-1,0,1\}$ as arguments.

For even $d$ and $i \in \{1,2\}$ \cite[Theorem 6.26]{LN97}:
\begin{align}
N(Q_i^d,\eta(-r)) &= (-1)^i q^{(d-2)/2} \text{ for $r \neq 0$, and} \label{eq:NevenDnonzeroR}\\
N(Q_i^d,0) &= (-1)^{i+1}(q^{d/2} - q^{(d-2)/2}). \label{eq:NevenDzeroR}
\end{align}
For $d$ odd ($i \in \{3,4\}$) \cite[Theorem 6.27]{LN97}:
\begin{align}
    N(Q_i^d,\eta(-r)) &= (-1)^{i+1}\eta(-r)q^{(d-1)/2} \text{ for $r \neq 0$, and} \label{eq:NoddDnonzeroR}\\
    N(Q_i^d,0) &= 0. \label{eq:NoddDzeroR}
\end{align}
The number of points in the intersection of a pair $s_1 \neq s_2$ of distinct $Q$-spheres depends on $Q$, the radii $r_1,r_2$ of the spheres, and the distance $t = Q(c_1-c_2)$ between the centers of the spheres, and is equal to 
$q^{d-2}+N_2(Q,r_1,r_2,t).$
The values of $N_2$ for various parameters are listed in \cref{tab:N_2EvenD,tab:N_2OddD}, where 
\[
D(t,r_1,r_2)=t^2+r_1^2+r_2^2-2(tr_1 + tr_2 + r_1r_2).
\]
Although the calculations to obtain these values are standard (see, for example \cite[Exercise 6.31]{LN97}), we include them in the appendix for the reader's convenience.

\begin{table}[h]
\begin{center}
\begin{tabular}{ccc||c}
\multicolumn{3}{c||}{}   & The values of $N_2(Q_i^d,r_1,r_2,t)$ for $i=1,2$\\ \hline \hline
\multicolumn{1}{c|}{\cref{eq:t0r}} & \multicolumn{1}{c|}{$t=0$} & \multicolumn{1}{c||}{$r_1=r_2=0$} & $(-1)^{i+1}(q^{d/2} - q^{\frac{d-2}{2}})$ \\ \hline
\multicolumn{1}{c|}{\cref{eq:t0r}} & \multicolumn{1}{c|}{$t=0$} & \multicolumn{1}{c||}{$r_1=r_2=r\neq 0$} & $(-1)^{i}q^{\frac{d-2}{2}}$ \\ \hline
\multicolumn{1}{c|}{\cref{eq:t0r1r2}} & \multicolumn{1}{c|}{$t=0$} & \multicolumn{1}{c||}{$r_1 \ne r_2$} &  $0$ \\ \hline
\multicolumn{1}{c|}{\cref{eq:D0}} & \multicolumn{1}{c|}{$t\ne 0$} & \multicolumn{1}{c||}{$D=0$} &  $0$ \\ \hline
\multicolumn{1}{c|}{\cref{eq:D1}} & \multicolumn{1}{c|}{$t \ne 0$} & \multicolumn{1}{c||}{$\eta(D)=1$} &  $(-1)^{i+1}q^{\frac{d-2}{2}}$ \\ \hline
\multicolumn{1}{c|}{\cref{eq:Dm1}} & \multicolumn{1}{c|}{$t \ne 0$} & \multicolumn{1}{c||}{$\eta(D)=-1$} & $(-1)^i q^{\frac{d-2}{2}}$ \\ 
\end{tabular}
\caption{The values of $N_2(Q_i^d,r_1,r_2,t)$ for $i=1,2$ when $d$ is even}\label{tab:N_2EvenD}
\end{center}
\end{table}

\begin{table}[h]
\begin{center}
\begin{tabular}{ccc||c}
\multicolumn{3}{c||}{}   & The values of $N_2(Q_i^d,r_1,r_2,t)$ for $i=3,4$\\ \hline \hline
\multicolumn{1}{c|}{\cref{eq:t0r}} & \multicolumn{1}{c|}{$t=0$} & \multicolumn{1}{c||}{$r_1=r_2=0$} & $0$ \\ \hline
\multicolumn{1}{c|}{\cref{eq:t0r}} & \multicolumn{1}{c|}{$t=0$} & \multicolumn{1}{c||}{$r_1=r_2=r\neq 0$} & $(-1)^{i+1}\eta(-r)q^{\frac{d-1}{2}}$ \\ \hline
\multicolumn{1}{c|}{\cref{eq:t0r1r2}} & \multicolumn{1}{c|}{$t=0$} & \multicolumn{1}{c||}{$r_1 \ne r_2$} &  $0$ \\ \hline
\multicolumn{1}{c|}{\cref{eq:D0}} & \multicolumn{1}{c|}{$t\ne 0$} & \multicolumn{1}{c||}{$D=0$} &  $(-1)^{i+1}\eta(-t)(q^{\frac{d-1}{2}}-q^{\frac{d-3}{2}})$ \\ \hline
\multicolumn{1}{c|}{\cref{eq:D1}} & \multicolumn{1}{c|}{$t \ne 0$} & \multicolumn{1}{c||}{$\eta(D)=1$} &  $(-1)^{i}\eta(-t)q^{\frac{d-3}{2}}$ \\ \hline
\multicolumn{1}{c|}{\cref{eq:Dm1}} & \multicolumn{1}{c|}{$t \ne 0$} & \multicolumn{1}{c||}{$\eta(D)=-1$} & $(-1)^{i}\eta(-t)q^{\frac{d-3}{2}}$ \\ 
\end{tabular}
\caption{The values of $N_2(Q_i^d,r_1,r_2,t)$ for $i=3,4$ when $d$ is odd}\label{tab:N_2OddD}
\end{center}
\end{table}

\section{New incidence bounds}\label{sec: new incidences}

\cref{th:smallSBound} follows directly from \cref{th:generalWeightedBound} using \cref{tab:N_2EvenD}.

\smallSBound*
\begin{proof}
    We will apply \cref{th:generalWeightedBound} for each equation separately.

    First, suppose that $i \in \{1,2\}$ and $S$ is a set of spheres of non-zero radius.
    Referring to \cref{eq:NevenDnonzeroR} and using the notation of \cref{th:generalWeightedBound}, we have $p^2q^d = q^{d-2} + (2+o(1))(-1)^{i}q^{d/2 - 2}$.
    For any pair of spheres $s_1 \neq s_2 \in S$, by \cref{tab:N_2EvenD} we have that $|s_1 \cap s_2| \leq q^{d-2} + q^{d/2 - 1}$.
    Hence, $\sum_{s_1 \neq s_2 \in S}(|s_1 \cap s_2| - p^2q^d) \leq (1+o(1))q^{d/2 - 1}|S|^2$.
    \Cref{eq:i2IncidenceBound} now follows directly from \cref{th:generalWeightedBound}.

    Now suppose that $i=2$ and $S$ is a set of spheres of zero radius.
    Referring to \cref{eq:NevenDzeroR}, we have $p^2q^d = q^{d-2} - (2+o(1))q^{d/2 - 1}$.
    Since $D(t,0,0)=t^2$, we have that $\eta(D) = 1$ if $t \neq 0$.
    Hence, referring to \cref{tab:N_2EvenD}, for any pair of spheres $s_1 \neq s_2 \in S$, we have $|s_1 \cap s_2| \leq q^{d-2} - q^{d/2 -1}$.
    Combining this with the bound on $p^2q^d$, this implies that $\sum_{s_1 \neq s_2 \in S}(|s_1 \cap s_2| - p^2q^d) \leq (1+o(1))|S|^2q^{d/2 - 1}$.
    \Cref{eq:i2ZeroIncidenceBound} now follows directly from \cref{th:generalWeightedBound}.

    Now suppose that $i=1$ and $S$ is a set of spheres of zero radius.
    By the same strategy used in the previous paragraphs, we have that $\sum_{s_1 \neq s_2 \in S} (|s_1 \cap s_2| - p^2q^d) \leq (1+o(1)) q^{d/2}|S|^2$, which leads directly to the claimed result.
\end{proof}
    The key difference between the proof of \cref{eq:i1IncidenceBound} and that of \cref{eq:i2IncidenceBound} and \cref{eq:i2ZeroIncidenceBound}, for $s_1 \neq s_2$ zero-radius $Q_1^d$-spheres whose centers are at distance zero, $|s_1 \cap s_2| = q^{d-2} + (1+o(1))q^{d/2}$.
    No surprise then that, in \cref{const:zeroRadEvenDim} (which shows that \cref{eq:i1IncidenceBound} is tight), all of the spheres have their centers at distance $0$.

Next, we give a new proof of a bound on the number of incidences between points and spheres of $0$ radius in odd dimensions, which is implicitly given in \cite{koh2015distance} (see Lemma 2.1 and Proposition 2.2 (1) in \cite{koh2015distance}). 
To prove this bound, we will need to use the fact that the entries in the fifth and sixth lines of \cref{tab:N_2OddD} depend on $\eta(-t)$, and hence we can obtain cancellation when the number of spheres is reasonably large.

\begin{lemma}\label{th:quadraticResidueDists}
    Let $C \subseteq \F_q^d$.
    Then,
    \[\left | \sum_{x,y \in C} \eta(Q(x-y)) \right | \leq \sqrt{2}q^{(d+1)/2}|C|. \]
\end{lemma}
\begin{proof}
    Let $S_+$ be the set of spheres centered at points of $C$ whose radii are squares, and let $S_-$ be the set of spheres centered at points of $C$ whose radii are non-squares.
    Then, using \cref{th:simpleWeightedBound} and the fact that $|S_+|,|S_-| = \frac{1}{2}|C|(q-1)$,
    \begin{align*}
        \left | \sum_{x,y \in C} \eta(Q(x-y)) \right | &= \left | \sum_{x,y \in C} \mathbf{1}_{\eta(Q(x-y))=1} - \sum_{x,y \in C} \mathbf{1}_{\eta(Q(x-y))=-1} \right | \\ &= \left | I(C,S_+) - I(C,S_-) \right | 
        \\ &\leq \sqrt{2}q^{(d+1)/2}|C|.
    \end{align*}
\end{proof}

\begin{theorem}\label{th:incidenceBoundOddD0R}
    Let $d$ be odd, let $i \in \{3,4\}$, let $P \subseteq \mathbb{F}_q^d$, and let $S$ be a set of $Q_i^d$-spheres with zero radius.
    Then,
    \[|I(P,S) - q^{-1}|P|\,|S|| < \left((1+\sqrt{2}) \, q^{d-1}|P||S|\right)^{1/2}. \]
\end{theorem}
\begin{proof}
    By \cref{eq:NoddDzeroR}, for each $s \in S$ we have $|s|=q^{d-1}$.
    Since $D(t,0,0)=t^2$, we have that $\eta(D(t,0,0)) = 0$ if and only if $t=0$.
    Using \cref{tab:N_2OddD}, 
    \begin{equation} \label{eq:initialBoundOddDZeroR}
        \sum_{s_1 \neq s_2 \in S} (|s_1 \cap s_2| - q^{d-2}) = q^{(d-3)/2} \sum_{s_1 \neq s_2 \in S} \eta(-Q(c_1-c_2)).
    \end{equation}
    Here, $c_1,c_2$ are centers of the spheres $s_1,s_2$, respectively.
    If $|S| < q^{(d+1)/2}$, then we bound the right side of \cref{eq:initialBoundOddDZeroR} by $q^{(d-3)/2}|S|^2 \leq q^{d-1}|S|$.
    If $|S| \geq q^{(d+1)/2}$, then we use \cref{th:quadraticResidueDists} to bound the right side of \cref{eq:initialBoundOddDZeroR} by $\sqrt{2}q^{(d-3)/2}q^{(d+1)/2}|S| \leq \sqrt{2}q^{d-1}|S|$.
    In either case, we have the bound
    \[\left| \sum_{s_1,s_2 \in S} (|s_1 \cap s_2| -q^{d-2})\right| \leq \sqrt{2}q^{d-1}|S|.  \]
    Using \cref{th:generalWeightedBound} with $p=q^{-1}$, we have
    \begin{align*}
        (I(P,S) - q^{-1}|P|\,|S|)^2 \leq (\sqrt{2}+1)q^d|P||S|,
    \end{align*}
    which immediately implies the conclusion of the theorem.
\end{proof}

Here comes the new proof of the odd-dimension case of \cref{thm:iosevichRudnev}.
Note that the constant factor we obtain is slightly weaker than that in \cref{thm:iosevichRudnev}.
We are not very careful to optimize this constant factor, but we cannot improve the bound of \cref{thm:iosevichRudnev} without some additional new idea.
It would be interesting to obtain a constant smaller than $2$.

Let $P \subseteq \mathbb{F}_q^d$, and let $S$ be a set of $Q_i^d$-spheres each having radius $r \neq 0$.
Denote
\[I^\Delta_r(n,m) = \max_{|P|=n,|S|=m}\left|I(P,S) - q^{-1}|P|\,|S|\right|.\]
Note that $I^\Delta_r(m,n) = I^\Delta_{r'}(m,n)$ if $\eta(r) = \eta(r')$, since the number of incidences between $P$ and $S$ is preserved by dilation.
\begin{theorem}\label{th:newIosevichRudnev}
    Let $d$ be odd, let $i \in \{3,4\}$, let $n,m$ be positive integers and $r \in \F_q^\times$.
    Then,
    \[I^\Delta_r(n,m) < 3q^{(d-1)/2}m^{1/2}n^{1/2}. \]
\end{theorem}
\begin{proof}
    We will use \cref{th:generalWeightedBound}.
    Let $P \subseteq \mathbb{F}_q^d$ with $|P|=n$, and let $S$ be a set of $Q_i^d$-spheres each having radius $r \neq 0$ with $|S|=m$.
    Let $j=(-1)^{i+1}\eta(-r)$, so that $N=jq^{(d-1)/2}$.
    Hence, $p=q^{-1}+jq^{-(d+1)/2}$, and $p^2q^d = q^{d-2} + 2j q^{(d-3)/2} + q^{-1}$.
    
    Since $D(t,r,r) = t(t-4r)$, we have $D(t,r,r)=0$ if and only if $t \in \{0,4r\}$.
    Referring to \cref{tab:N_2OddD}, the number of points in the intersection of two radius $r$ spheres at distance $0$ is $jq^{(d-1)/2}$, and the number of points in the intersection of two radius $r$ spheres at distance $4r$ is $j(q^{(d-1)/2}-q^{(d-3)/2}) = jq^{(d-1)/2} + (-1)^iq^{(d-3)/2}\eta(-4r)$.
    
    Let $S_0$ be the set of spheres having the same centers as the spheres of $S$ and radius $0$, and let $S_{4r}$ be the set of spheres having the same centers as the spheres of $S$ and radius $4r$.
    Let $C$ be the set of centers of spheres in $S$.
    
    Referring to \cref{tab:N_2OddD} and using the above observations about spheres at distance $0$ or $4r$, 
    \begin{align}
        \sum_{s_1 \neq s_2 \in S}(&|s_1 \cap s_2| - p^2q^d) \nonumber \\
        &\leq \sum_{s_1 \neq s_2 \in S}(|s_1 \cap s_2| - q^{d-2}) - 2jq^{(d-3)/2}m(m-1) \nonumber \\
        &= jq^{(d-1)/2}(I(C,S_0)+I(C,S_{4r}))- 2jq^{(d-3)/2}m(m-1) \nonumber  \\ & \hspace{15mm}+ (-1)^{i}q^{(d-3)/2}\sum_{x,y \in C}\eta(-Q(x-y)). \label{eq:sIntersectionOdd}
    \end{align}
    
    Taking absolute values,
    \begin{align*}
    \left | \sum_{s_1 \neq s_2 \in S}(\right. \left. \vphantom{\sum_{s_1 \neq s_2 \in S}}|s_1 \cap s_2| - p^2q^d) \right| \leq &\left|q^{(d-1)/2}(I(C,S_0)+I(C,S_{4r}))- 2q^{(d-3)/2}m^2 \right| \\ &+ \left|q^{(d-3)/2}\sum_{x,y \in C}\eta(-Q(x-y)) \right| + 2q^{(d-3)/2}m. 
    \end{align*}
    From the definition of $I^\Delta$ and the fact that $I^\Delta_{4r}=I^{\Delta}_r$, we have
    \[\left|(I(C,S_0) + I(C,S_{4r}) - 2q^{-1}m^2 \right| \leq I^\Delta_0(m,m) + I^\Delta_r(m,m). \]
    Hence, by \cref{th:quadraticResidueDists} and \cref{th:incidenceBoundOddD0R}, we have
    \begin{align*}
        \left | \sum_{s_1 \neq s_2 \in S}(\right. & \left. \vphantom{\sum_{s_1 \neq s_2 \in S}}|s_1 \cap s_2| - p^2q^d) \right| \\ 
        &\leq q^{(d-1)/2}\left(I^\Delta_0(m,m) + I^\Delta_r(m,m)\right) + \left |q^{(d-3)/2}\sum_{x,y \in C}\eta(-Q(x-y)) \right| + 2q^{(d-3)/2}m\\
        &< q^{(d-1)/2}\left(2q^{(d-1)/2}m +  I^\Delta_r(m,m)\right) + q^{(d-3)/2}\sqrt{2}q^{(d+1)/2}m + 2q^{(d-3)/2}m \\
        &< q^{(d-1)/2}I_r^\Delta(m,m) + 4q^{d-1}m.
    \end{align*}

    We first assume that $n=m$.
    Using \cref{th:generalWeightedBound}, we have
    \begin{align*}
    I_r^\Delta(m,m)^2 &\leq m(q^{(d-1)/2}I_r^\Delta(m,m) + 4q^{d-1}m)+m^2(q^{d-1}+N) \\
    &< m(q^{(d-1)/2}I_r^\Delta(m,m) + 6q^{d-1}m),
    \end{align*}
    so
    \[I_r^\Delta(m,m) < 3 q^{(d-1)/2}m.\]
    
    Now, we apply this bound to the general case:
    \begin{align*}
        I_r^\Delta(n,m)^2 &< n(q^{(d-1)/2}I_r^\Delta(m,m) + 6q^{d-1}m) \\
        &\leq 9q^{d-1}nm.
    \end{align*}
\end{proof}

Here comes the proof of \cref{th:ourOddBound}.
A key observation that enables this proof to work is in the following lemma.

\begin{lemma}\label{th:D=0ImpliesEqualEta}
    If $t,r_1,r_2 \neq 0$ and $D(t,r_1,r_2)=0$, then $\eta(t) = \eta(r_1) = \eta(r_2)$.
\end{lemma}
\begin{proof}
    The discriminant of $D$ with respect to $t$ is $16r_1r_2$.
    Hence, for fixed $r_1,r_2$, we have that $D(t,r_1,r_2) = 0$ has two solutions if $\eta(r_1)=\eta(r_2)$, and has zero solutions otherwise.
    Since $D$ is a symmetric polynomial, the same reasoning applies to show that $D(t,r_1,r_2) = 0$ only if $\eta(t) = \eta(r_1)$.
\end{proof}

\ourOddBound*
\begin{proof}
    By \cref{th:D=0ImpliesEqualEta}, if $D(t,r_1,r_2) = 0$ for $r_1,r_2 \in R$, then $(-1)^{i+1}\eta(-t) = (-1)^{i+1}\eta(-r_1) = -1$.
    Consequently, the contribution to $\sum_{s_1 \neq s_2 \in S}(|s_1 \cap s_2| - q^{d-2})$ from the fourth line of \cref{tab:N_2OddD} is negative, and can be ignored.
    Letting $p = q^{-1} - q^{-(d+1)/2}$ and referring to \cref{tab:N_2OddD}, we have
    \begin{align*}
    \sum_{s_1 \neq s_2 \in S}(|s_1 \cap s_2| - p^2q^d) &\leq \sum_{s_1 \neq s_2 \in S}(|s_1 \cap s_2| - q^{d-2}) + 2q^{(d-3)/2}|S|^2 \\
    &\leq 3q^{(d-3)/2}|S|^2.
    \end{align*}
    Hence, by \cref{th:generalWeightedBound},
    \[\left|I(P,S) - q^{-1}|P|\,|S|\right| \leq q^{(d-1)/2}|P|^{1/2}|S|^{1/2} + \sqrt{3}q^{(d-3)/4}|P|^{1/2}|S|, \]
    as claimed.
\end{proof}
The reason this proof doesn't give any new result for arbitrary radii is that, in the case $(-1)^{i+1}\eta(-r)=1$, we have no way to control the contribution from the fourth line of \cref{tab:N_2OddD} to $\sum_{s_1 \neq s_2 \in S}(|s_1 \cap s_2| - q^{d-2})$ that leads to any improvement over \cref{th:iosevichRudnevMultipleRadii}.
    
\section*{Acknowledgments}
D. Koh was supported by the National Research Foundation of Korea (NRF) grant funded by the Korea government (MSIT) (NO. RS-2023-00249597).
B. Lund and S. Yoo were supported by the Institute for Basic Science (IBS-R029-C1). 


\bibliographystyle{plain}
\bibliography{FFSphereIncidence}

\appendix
\section{Pairwise intersection sizes of spheres}
\label{appendix}

In this section, we compute the number of points in the intersection of any two spheres.
Since the orthogonal group associated to a quadratic form $Q$ acts transitively on the level sets of $Q$, the number of points in the intersection of two spheres depends only on the radii of the spheres, and the distance between their centers.
The size of this intersection is always $q^{d-2}$ plus some lower order terms, which we denote by
\[N_2(Q,r_1,r_2,t) = \#\{x: Q(x-y_1) = r_1, Q(x-y_2) = r_2, Q(y_1-y_2) = t\} - q^{d-2}. \]
In other words, $N_2(Q,r_1,r_2,t)$ is the difference between $q^{d-2}$ and the number of points in the intersection of two $Q$-spheres, of $Q$-radius $r_1$ and $r_2$ respectively, when the $Q$-distance between the centers of the spheres is $t$.

Let $d \geq 3$.
Fix a choice of $Q_i^d$ and $t \in \mathbb{F}_q$.
Let $y = (1,t,0,\ldots,0)$.
Each point $x$ in the intersection of the radius $r_1$ sphere centered at the origin with the radius $r_2$ sphere centered at $y$ satisfies 
\begin{align*}
Q_i^d(x) &= x_1x_2 + Q_i^{d-2}(x_3,\ldots,x_d) = r_1 \text{, and} \\
Q_i^d(x-y) &= (x_1 -1)(x_2 -t) + Q_i^{d-2}(x_3,\ldots,x_d) =  r_2.
\end{align*}
Subtracting the second equation from the first and rearranging, we find
\[x_2 = r_1 - r_2 + t - tx_1.  \]
Substituting this into the equation $Q_i^d(x) = r_1$, we find that
\[Q_i^{d-2}(x_3,\ldots,x_d) = r_1 -x_1x_2 = tx_1^2 - (r_1 - r_2 +t)x_1 + r_1. \]
For $x \in \F_q$, let $f(x) = tx^2 - (r_1 - r_2 + t)x + r_1$.
We have
 \[N_2(Q_i^d, r_1, r_2, t) = \sum_{x \in \F_q} N(Q_i^{d-2}, \eta(-f(x))). \]
If $t=0$, then $f(x) = r_1 - (r_1-r_2)x$.
If additionally $r_1=r_2=r$, then $f(x) = r$, and so
\begin{equation}\label{eq:t0r}
N_2(Q_i^d, r, r, 0) = q N(Q_i^{d-2}, \eta(-r)). \end{equation}
If $t=0$ and $r_1 \neq r_2$, then the image of $f(x)$ is $\F_q$ and so
\begin{small}
\begin{equation}\label{eq:t0r1r2}
N_2(Q_i^d, r_1, r_2, 0) = \sum_{z \in \F_q} N(Q_i^{d-2}, \eta(-z)) = N(Q_i^{d-2}, 0) + \frac{q-1}{2}(N(Q_i^{d-2}, -1) + N(Q_i^{d-2}, 1)).
\end{equation}
\end{small}
Now, suppose that $t \neq 0$.
Then, setting $z=2tx-(r_1 -r_2 +t)$, and $D = (r_1 - r_2 + t)^2 - 4r_1t$, we have
\[t^{-1} f(x) = \left(x - \frac{r_1 - r_2 + t}{2t}\right)^2 - \frac{(r_1 - r_2 +t)^2-4r_1t}{4t^2} = \frac{z^2 - D}{4t^2}. \]
Hence, $\eta(-f(x)) = \eta(z^2 - D) \eta(-t)$.

If $D = 0$, then $\eta(-f(x)) = \eta(-t)$ except when $z=0$, so
\begin{equation}\label{eq:D0}N_2(Q_i^d, r_1, r_2, t) = (q-1) N(Q_i^{d-2}, \eta(-t))+N(Q_i^{d-2}, 0). \end{equation}
If $\eta(D)=1$, then there are two $z$ such that $f(x) = \frac{z^2 - D}{4t} = 0$.
Using the fact that $\sum_{z} \eta(z^2 - D) = -1$ (see, for example \cite[Theorem 5.48]{LN97}), this implies that
\begin{equation}\label{eq:D1}N_2(Q_i^d, r_1, r_2, t) = 2 N(Q_i^{d-2}, 0) + \frac{q-1}{2}N(Q_i^{d-2}, -\eta(-t)) + \frac{q-3}{2}N(Q_i^{d-2}, \eta(-t)).
\end{equation}
If $\eta(D)=-1$, then there is no $z$ such that $f(x)=\frac{z^2-D}{4t} = 0$, and so
\begin{equation}\label{eq:Dm1}N_2(Q_i^d, r_1, r_2, t) = \frac{q+1}{2}N(Q_i^{d-2}, -\eta(-t)) + \frac{q-1}{2}N(Q_i^{d-2}, \eta(-t)).\end{equation}

\end{document}